\documentclass[11pt]{article}
\usepackage{cite,geometry,graphicx,amsmath,amssymb,epstopdf,amsthm,mathrsfs,extarrows,mathtools,stmaryrd}  
\usepackage{tikz-cd,tikz}
\usetikzlibrary{matrix,arrows,decorations.pathmorphing}
\DeclareMathOperator{\pet}{\mathbf{P}_\textnormal{\emph{e}}}
\DeclareMathOperator{\px}{\mathbf{P}_\textnormal{\emph{x}}}

\newcommand{\cent}{\overline{G^e}}
\newcommand*{\longhookrightarrow}{\ensuremath{\lhook\joinrel\relbar\joinrel\rightarrow}}
\newtheorem{lem}{Lemma}
\numberwithin{lem}{section}
\numberwithin{equation}{section}
\newtheorem{thm}[lem]{Theorem}
\newtheorem{prop}[lem]{Proposition}
\theoremstyle{definition}
\newtheorem{definition}[lem]{Definition}
\newtheorem*{rk*}{Remark}
\newtheorem{rk}[lem]{Remark}

\title{The Peterson Variety and the Wonderful Compactification}
\author{Ana B\u{a}libanu}
\date{}

\begin{document}
\maketitle

\begin{abstract}
We look at the centralizer in a semisimple algebraic group $G$ of a regular nilpotent element $e\in\text{Lie}(G)$, and show that its closure in the wonderful compactification is isomorphic to the Peterson variety. It follows that the closure in the wonderful compactification of the centralizer $G^x$ of any regular element $x\in\text{Lie}(G)$ is isomorphic to the closure of a general $G^x$-orbit in the flag variety. We also give a description of the $G^e$-orbit structure of the Peterson variety.
\end{abstract}

\section{Introduction}
The wonderful compactification of a semisimple complex algebraic group $G$ of adjoint type is a special case of the compactification of symmetric spaces introduced by DeConcini and Procesi in \cite{decproc}. Its boundary is a divisor with normal crossings with a unique closed $G\times G$-orbit, and in some sense it encodes the behavior of the group ``at infinity''. A survey of its structure can be found in \cite{evens}.

We will consider regular elements in the Lie algebra $\mathfrak{g}=\text{Lie}(G)$ and their centralizers in $G$, and describe the closure of these centralizers in the wonderful compactification $\overline{G}$. In particular, we will be interested in the unique conjugacy class of regular nilpotent elements, also called principal nilpotents. All the relevant structure theory of semisimple Lie algebras and of their regular orbits was developed by Kostant in \cite{kost1} and \cite{kost2}.

A principal nilpotent element sits inside a unique Borel subgroup $B$, and its centralizer is a unipotent abelian subgroup of $B$. In the full flag variety determined by the opposite Borel the closure of a general orbit of this centralizer is called the Peterson variety. This variety has been well-studied, and is known to be singular and non-normal except in very small rank \cite{kost3}. It was introduced by Dale Peterson in the 1990s and it has proved essential in the study of the quantum cohomology of flag varieties, for example in \cite{kost3}, \cite{rietsch}, and \cite{tym}.

We will show that the closure of the centralizer of the principal nilpotent in $\overline{G}$ is isomorphic to the Peterson variety. This will lead to the main result of this paper, which states that the closure in $\overline{G}$ of the centralizer of any regular element of $\mathfrak{g}$ is isomorphic to the closure of a sufficiently general orbit of this centralizer in the flag variety. Both of these results are shown by choosing appropriate projective embeddings given by very ample line bundles, and then establishing an isomorphism between the resulting homogeneous coordinate rings. 

This extends the case of a maximal torus $T$, which is the centralizer of a regular semisimple element---the closure of $T$ in the wonderful compactification is the toric variety whose fan is the fan of Weyl chambers (see \cite{evens}, Remark 4.5), and it is isomorphic to the closure of a general $T$-orbit in the flag variety \cite{dab}. 

In Section \ref{peterson} we recall some basics about the flag variety $\mathcal{B}$, the Peterson variety, and the very ample $G$-equivariant line bundles on $\mathcal{B}$. In Section \ref{wond} we present some analogous facts about the wonderful compactification by describing its construction via the Vinberg semigroup. In Section \ref{main} we construct an isomorphism between the homogeneous coordinate rings of the closure of the regular nilpotent centralizer in $\overline{G}$ and the Peterson variety. In Section \ref{centralizers} we extend the results of Section \ref{main} to the case of the centralizer of an arbitrary regular element. In Section \ref{orbit} we give an explicit description of the orbits of the regular nilpotent centralizer on the Peterson variety, from which it becomes clear, in particular, that except in very few cases they are infinite in number.

The author would like to thank Victor Ginzburg, her Ph.D. advisor, for his advice and guidance throughout the project, and Michel Brion and Sam Evens for helpful suggestions and discussions.

\section{The Peterson Variety}
\label{peterson}
Let $G$ be, as above, a complex semisimple algebraic group of adjoint type and rank $l$, $T$ a maximal torus, and $B$ a Borel subgroup containing $T$. Let $N$ be the unipotent radical of $B$ and $\alpha_1,\ldots,\alpha_l$ the set of positive simple roots. If $e_1,\ldots,e_l$ are corresponding simple root vectors in the Lie algebra $\mathfrak{g}$, then 
\[e=e_1+\ldots+e_l\]
is a principal nilpotent sitting inside $\mathfrak{b}$, and we denote by $G^e$ its centralizer in $G$. 

The centralizer $G^e$ is a unipotent abelian subgroup of $N$ of dimension equal to $l$. In type $A$, $e$ is the single nilpotent Jordan block, and $G^e$ is the group of unipotent matrices with constant entries along each superdiagonal.

Let $B^-$ be the opposite Borel subgroup and $\mathfrak{b}^-$ its Lie algebra. Viewing Borel subalgebras as points in the flag variety, $\mathfrak{b}^-$ is the basepoint of the flag variety $\mathcal{B}=G/B^-$, and the Peterson variety is the closure
\[\pet:=\overline{G^e\cdot\mathfrak{b}^-}\subset \mathcal{B}.\]

Recall that the $G$-equivariant line bundles on $\mathcal{B}$ are indexed by integral weights of $\mathfrak{g}$, and for a dominant weight $\lambda$ the space of global sections of the line bundle $\mathcal{L}_\lambda$ is identified with $V_\lambda$, the irreducible representation of highest weight $\lambda$, via
\begin{align*}
V_\lambda&\xlongrightarrow{\sim}\Gamma(\mathcal{B},\mathcal{L}_\lambda)\\
		v&\longmapsto\left[gB^-\mapsto (gB^-,v_\lambda^*(g^{-1}\cdot v))\right],
		\end{align*} 
where $v_\lambda^*$ is the lowest weight vector of $V^*_\lambda$. (Note that the space of global sections is $V_\lambda$ and not its dual, because we are taking the flag variety relative to the opposite Borel $B^-$.) 

Let $\omega_1,\ldots,\omega_l$ be the fundamental weights of $\mathfrak{g}$, $V_1,\ldots, V_l$ the fundamental representations, and for each $i$, let $V_i^*$ be the dual representation with lowest weight vector $v_i^*$. The Pl\"{u}cker embedding realizes the flag variety as a multi-projective variety
\begin{align*}
\mathcal{B}&\longhookrightarrow\prod_{i=1}^l\mathbb{P}(V_i^*)\\
gB^-&\longmapsto (g\cdot [v_1^*],\ldots,g\cdot[v_l^*])
\end{align*}
and its total coordinate ring is the multi-graded algebra given by summing the spaces of global sections of all $G$-equivariant line bundles:
\[R[\mathcal{B}]:=\bigoplus_{\lambda\text{ dom}}\Gamma(\mathcal{B},\mathcal{L}_\lambda)=\bigoplus_{\lambda\text{ dom}}V_\lambda.\]
(See \cite{arzh} for a detailed introduction to total coordinate rings, also called Cox rings.) Multiplication is given by projection onto the highest weight component:
\[V_{\lambda}\otimes V_{\mu}\longrightarrow V_{\lambda+\mu}.\]

The very ample line bundles on $\mathcal{B}$ correspond to regular dominant weights, and such a weight $\lambda$ produces a $\mathbb{Z}$-graded homogeneous coordinate ring, denoted $R_\lambda[\mathcal{B}]$, that is a quotient of the total coordinate ring given by taking a generic line in the semigroup of dominant weights:
\[R_\lambda[\mathcal{B}]:=\bigoplus_{n\geq0}\Gamma(\mathcal{B},\mathcal{L}_{n\lambda})=\bigoplus_{n\geq0} V_{n\lambda}.\]
The homogeneous coordinate ring of $\pet$ is then 
\[R_\lambda[\pet]=R_\lambda[\mathcal{B}]/\mathcal{I}_{\pet},\]
where 
\begin{align*}
\mathcal{I}_{\pet}&=\bigoplus_{n\geq0}\left\{u\in V_{n\lambda}\mid v_{n\lambda}^*(g\cdot u)=0,\, \forall gB^-\in\pet\right\}\\
	&=\bigoplus_{n\geq0}\left\{u\in V_{n\lambda}\mid v_{n\lambda}^*(g\cdot u)=0,\, \forall g\in G^e\right\}
\end{align*} is the ideal of global sections that vanish on the Peterson variety.

\section{The Wonderful Compactification}
\label{wond}
Let $\widetilde{G}$ be the simply-connected cover of $G$, $\widetilde{T}$ the corresponding maximal torus, and $\widetilde{Z}$ its center. Identifying $\text{End}(V_i)$ with $V_i\otimes V_i^*$ gives representation maps 
\[\rho_i:\widetilde{G}\longrightarrow V_i\otimes V_i^*.\]
 
We recall briefly the construction of the wonderful compactification via the Vinberg semigroup \cite{vinb}. Consider $\widetilde{G}\times_{\widetilde{Z}}\widetilde{T}$, where $\widetilde{Z}\longhookrightarrow \widetilde{G}\times \widetilde{T}$ is the anti-diagonal embedding. Define the embedding 
\begin{align*}
\chi:\, \widetilde{G}\times_{\widetilde{Z}}\widetilde{T}&\longhookrightarrow \mathbb{C}^l\times \prod_{i=1}^lV_i\otimes V_i^*\\
(g,t)&\longmapsto (\alpha_1(t),\ldots,\alpha_l(t),\omega_1(t)\rho_1(g),\ldots,\omega_l(t)\rho_l(g))
\end{align*}
The closure of the image of $\chi$ is the Vinberg semigroup $V_G$, and the first projection is a flat family of semigroups over $\mathbb{C}^l$ (see \cite{vinb}, Section 4.) The closure $V_G^0$ of the image of $\chi$ in the space
\[\mathbb{C}^l\times \prod_{i=1}^l(V_i\otimes V_i^*-\{0\})\]
 is a smooth open subset. Since $\tilde{Z}$ is central, $\widetilde{G}\times_{\widetilde{Z}}\widetilde{T}$ is a group, and it acts naturally on both $V_G$ and $V_G^0$. In particular, the torus $\{1\}\times \widetilde{T}$ acts freely on $V_G^0$ via coordinate-wise multiplication by 
\[(\alpha_1(t),\ldots,\alpha_l(t),\omega_1(t),\ldots,\omega_l(t)),\]
and the wonderful compactification of $G$ is defined to be the quotient of $V_G^0$ by this action:
\[\overline{G}:=V_G^0/\widetilde{T}\]
(see \cite{marthad}, 5.3.) It contains $G\cong\widetilde{G}/\widetilde{Z}$ as a dense open subset, and it has a natural $\widetilde{G}\times\widetilde{G}$-action on $\overline{G}$ that extends the two-sided action of $\widetilde{G}$ on $G$ itself.

The $\widetilde{G}\times \widetilde{G}$-equivariant line bundles on $\overline{G}$ correspond to integral weights of the group $\widetilde{G}$. For such a weight $\lambda$, the global sections of the line bundle $\mathcal{M}_\lambda$ are given by
\[\Gamma(\overline{G},\mathcal{M}_\lambda)\cong\bigoplus_{\mu\leq\lambda} V^*_\mu\otimes V_\mu\]
as a $\widetilde{G}\times \widetilde{G}$-module, where the sum is over all dominant weights $\mu$ less than $\lambda$---i.e. dominant weights $\mu$ such that $\lambda-\mu$ is a sum of simple roots with non-negative integral coefficients (see \cite{brion}, 3.2.3.) The line bundle $\mathcal{M}_\lambda$ is very ample exactly when $\lambda$ is a regular dominant weight.

The total coordinate ring of $\overline{G}$---that is, the affine coordinate ring of the Vinberg semigroup---is the multi-graded algebra
\[R[\overline{G}]=\bigoplus_{\lambda}\Gamma(\overline{G},\mathcal{M}_\lambda)=\bigoplus_{\lambda}\left(\bigoplus_{\mu\leq\lambda} V^*_\mu\otimes V_\mu\right)t^\lambda,\]
with multiplication on the right hand side given by viewing the algebra as a subalgebra of $\mathbb{C}[\widetilde{G}\times \widetilde{T}]$. In particular, the multiplication map has the property
\[m:(V^*_\mu\otimes V_\mu)\otimes (V^*_\nu\otimes V_\nu)\cong (V_\mu\otimes V_\nu)^*\otimes (V_\mu\otimes V_\nu)\longrightarrow\bigoplus_{\xi\leq\mu+\nu}V^*_\xi\otimes V_\xi\]
by decomposing $(V_\mu\otimes V_\nu)^*$ and $V_\mu\otimes V_\nu$ separately into irreducible representations and then projecting onto the components of the form $V_\xi^*\otimes V_\xi$.

From now on fix a regular dominant weight $\lambda$ in the root lattice. Then for any $\mu\leq\lambda$ the $\widetilde{G}$-representations $V_\lambda$ and $V_\mu$ descend to representations of the adjoint group $G$. The $\mathbb{Z}$-graded homogeneous coordinate ring of $\overline{G}$ produced by the very ample line bundle $\mathcal{M}_\lambda$ is a quotient algebra of $R[\overline{G}]$ corresponding to the generic line given by $\lambda$ in the cone of dominant weights:
$$R_\lambda[\overline{G}]:=\bigoplus_{n\geq0}\Gamma(\overline{G},\mathcal{M}_{n\lambda})=\bigoplus_{n\geq0}\left(\bigoplus_{\mu\leq n\lambda} V^*_\mu\otimes V_\mu\right)t^{n\lambda}.$$

Let $\cent$ be the closure of the centralizer of the principal nilpotent in the wonderful compactification of $G$. Its homogeneous coordinate ring is then
\[R_\lambda[\overline{G^e}]=R_\lambda[\overline{G}]/\mathcal{I}_{\cent},\]
where 
\begin{align*}
\mathcal{I}_{\cent}&=\bigoplus_{n\geq0}\left\{\sum f^\mu_{v^*,u}t^{n\lambda}\in \bigoplus_{\mu\leq n\lambda} V_{\mu}^*\otimes V_{\mu}\mid \sum f^\mu_{v^*,u}(g)\lambda(t)^n=0,\,\forall (g,t)\in G^e\times T\right\}\\
	&=\bigoplus_{n\geq0}\left\{\sum f^\mu_{v^*,u}t^{n\lambda}\in \bigoplus_{\mu\leq n\lambda} V_{\mu}^*\otimes V_{\mu}\mid \sum f^\mu_{v^*,u}(g)=0,\,\forall g\in G^e\right\},
\end{align*}
is the homogeneous ideal of global sections vanishing on $\cent$.

\begin{rk}
\label{choices}
Because of the choice of $\lambda$ above, from now on whenever a representation $V_\mu$ with highest weight $\mu$ appears, the weight $\mu$ will be an element of the root lattice, and $V_\mu$ will descend to a representation of $G$. 

This choice is not necessary, and the same argument goes through essentially unchanged with an arbitrary choice of regular dominant $\lambda$---however, this will allow us to apply $G$ directly to the spaces $V_\mu$ without having to repeatedly refer to the simply-connected cover $\widetilde{G}$.

The following lemmas and propositions use only the fact that $G$ is semisimple, that $\lambda$, $\mu$, $\nu$ are weights of $G$, and that $G^e$ is an abelian unipotent subgroup of $G$ that centralizes the principal nilpotent $e$. Therefore they will apply also to the setting of Section \ref{centralizers}, where the group under consideration will not necessarily be of adjoint type.
\end{rk}

Before we begin to prove our results, we introduce some notation: For any dominant weight $\mu$ in the root lattice, and any $u\in V_\mu$ and $v^*\in V_\mu^*$, denote by $f^\mu_{v^*,u}$ the function of $G$ corresponding to the matrix entry $v^*\otimes u\in V_\mu^*\otimes V_\mu$---that is,
\[f^\mu_{v^*,u}(g)=v^*(g\cdot u).\]
Let $v_\mu^*$ denote the lowest weight vector of $V_\mu^*$, and make this choice such that, under the multiplication map
\[V_\mu^*\otimes V_\nu^*\longrightarrow V_{\mu+\nu}^*,\]
$v_{\mu+\nu}^*$ is the image of $v_\mu^*\otimes v_\nu^*$, for any dominant weights $\mu$ and $\nu$. (This can be done inductively, beginning from the fundamental representations.) Then, since 
\[v_\mu^*\otimes v_\nu^*\in V_\mu^*\otimes V_\nu^*\]
always belongs to the irreducible component of the tensor isomorphic to $V_{\mu+\nu}^*$, we have
\[m(v_\mu^*\otimes u_1,v_\nu^*\otimes u_2)=v_{\mu+\nu}^*\otimes u\in V_{\mu+\nu}^*\otimes V_{\mu+\nu}\]
where $u$ is the projection of the tensor $u_1\otimes u_2\in V_\mu\otimes V_\nu$ onto the irreducible component $V_{\mu+\nu}$. In other words,
\begin{align}
\label{topmult}
f^\mu_{v_\mu^*,u_1}\cdot f^\nu_{v_\nu^*,u_2}=f^{\mu+\nu}_{v^*_{\mu+\nu},u}.
\end{align}

\section{The Principal Nilpotent Case}
\label{main}
In this section we will show that the varieties $\pet$ and $\cent$ are isomorphic, by establishing an isomorphism between the homogeneous coordinate rings $R_\lambda[\pet]$ and $R_\lambda[\overline{G^e}]$. Define, component-wise, a map 
\begin{align}
\label{phi'}
\Phi': \,R_\lambda[\mathcal{B}]&\longrightarrow R_\lambda[\overline{G}]\\
	u&\longmapsto (v_{n\lambda}^*\otimes u)t^{n\lambda}\nonumber
	\end{align}
for $u\in V_{n\lambda}$. We will show

\begin{thm}
\label{mainthm}
The map $\Phi'$ descends to an isomorphism of graded algebras
\[\Phi: R_\lambda[\pet]\longrightarrow R_\lambda[\cent].\]
\end{thm}

\begin{rk}
The argument that follows can be applied directly to the multi-graded total coordinate rings as well, but this approach is significantly more technical. Choosing a suitable $\mathbb{Z}$-graded homogeneous coordinate ring for each variety circumvents these technicalities.
\end{rk}

\begin{lem}[\cite{ginz}, Corollary 1.6]
\label{ann}
 For any vector $v^*\in V_\mu^*$, one has 
 \[\text{\emph{Ann}}_{\mathcal{U}\mathfrak{g}^e}(v_\mu^*)\subseteq \text{\emph{Ann}}_{\mathcal{U}\mathfrak{g}^e}(v^*).\]
\end{lem}

\begin{rk}
This lemma follows from the following result of Ginzburg on the cohomology of the loop Grassmannian. We were unable to find a direct algebraic proof in the literature.
\end{rk}

\begin{thm} [\cite{ginz}, Theorem 1.5] 
\label{affgrass} Let $\mathbb{O}_\lambda$ be the orbit of the affine Grassmannian of the Langlands dual $\check{G}$ of $G$ corresponding to the dominant weight $\lambda$ of $G$. Then there is a natural isomorphism of graded algebras
\begin{align}
H^{\bullet}(\overline{\mathbb{O}}_{\lambda},\mathbb{C})\simeq\mathcal{U}\mathfrak{g}^e/\text{\emph{Ann}}_{\mathcal{U}\mathfrak{g}^e}(v_\lambda^*).
\end{align}
\end{thm}

\begin{prop}
\label{toprow}
Let $v^*\otimes u\in V_\mu^*\otimes V_\mu$. Then there exists an element $w\in V_\mu$ such that for any $g\in G^e$, 
\[v^*(g\cdot u)=v_\mu^*(g\cdot w).\]
\end{prop}
\begin{proof}
We will first show this for linear functions on the universal enveloping algebra $\mathcal{U}\mathfrak{g}^e$ of the nilpotent abelian subalgebra $\mathfrak{g}^e=\text{Lie}(G^e)$ of $\mathfrak{g}.$ 

Let $v_1,\ldots, v_r$ be a basis of weight vectors for $V_\mu$, and let $v_1^*,\ldots,v_r^*$ be the dual basis for $V_\mu^*$. One can make this choice so that $v_1^*=v_\mu^*$. Then the representation map is
\begin{align*}
\varphi:\mathcal{U}\mathfrak{g}^e&\longrightarrow V_\mu\otimes V_\mu^*\\
x&\longmapsto\sum v_i^*(x\cdot v_j)v_i\otimes v_j^*.
\end{align*}

 Let $\{e,h,f\}$ be a principal $\mathfrak{sl}_2$-triple in $\mathfrak{g}$---this triple is unique up to conjugation by $G^e$, and the element $h$ is regular and semisimple, with $[h,e]=2e$ (see \cite{kost1}.) Then $\mathfrak{g}$, the universal enveloping algebra $\mathcal{U}\mathfrak{g}$, and the vector space $V_\mu$ all have natural $\mathbb{Z}$-gradings by the eigenvalues of $h$, and $\mathfrak{g}^e$ sits in strictly positive degrees. 

The map 
\[\mathcal{U}\mathfrak{g}^e\longrightarrow\text{End}(V_\mu)\]
is $\mathbb{Z}$-graded---if $x\in\mathcal{U}\mathfrak{g}$ has degree $m$, and $v\in V_\mu$ has degree $k$, then $x\cdot v\in V_\mu$ has degree $m+k$. Therefore, for any $m$ greater than the maximum eigenvalue $M$ of $h$ on $\text{End}(V_\mu)=V_\mu\otimes V_\mu^*$, we have
\[\varphi\vert_{\mathcal{U}^{ m}\mathfrak{g}^e}=0,\]
where $\mathcal{U}^{m}\mathfrak{g}^e$ denotes the component of $\mathcal{U}\mathfrak{g}^e$ of degree $m$.
So without loss of generality we can restrict to considering
\[\varphi:\mathcal{U}^{\leq M}\mathfrak{g}^e\longrightarrow V_\mu\otimes V_\mu^*.\]
Since all of these spaces are finite-dimensional, we will be able to dualize without issue.

The dual map $\varphi^*:V_\mu^*\otimes V_\mu\longrightarrow (\mathcal{U}^{\leq M}\mathfrak{g}^e)^*$ realizes the elements of $V_\mu^*\otimes V_\mu$ as functions on the universal enveloping algebra, via
\[\varphi^*(v_i^*\otimes v_j)(x)=v_i^*(x\cdot v_j)\quad\text{ for any } x\in \mathcal{U}^{\leq M}\mathfrak{g}^e.\]
Consider the commutative diagram
\begin{equation}
\label{first}
\begin{tikzcd}
  V_\mu^*\otimes V_\mu\arrow{r}{\varphi^*}&(\mathcal{U}^{\leq M}\mathfrak{g}^e)^*\\
 v_\mu^*\otimes V_\mu\arrow[hook]{u}{}\arrow{ur}{\psi^*}
\end{tikzcd}
\end{equation}
where the vertical map is the inclusion induced by 
\[\mathbb{C}v_\mu^*\longhookrightarrow V_\mu^*,\]
and $\psi^*$ is the restriction of $\varphi^*$ to the subspace $v_\mu^*\otimes V_\mu$. 

We would like to first show that every function in $V_\mu^*\otimes V_\mu$ on $\mathcal{U}^{\leq M}\mathfrak{g}^e$ comes from a function in $v_\mu^*\otimes V_\mu$---that is, that the image of $\varphi^*$ is equal to the image of $\psi^*$, or in other words that  
\[\text{coker}(\varphi^*)=\text{coker}(\psi^*).\]
Dualizing diagram \eqref{first}, we obtain
\begin{equation}
\label{second}
\begin{tikzcd}
  V_\mu\otimes V_\mu^*\arrow{d}{}&\mathcal{U}^{\leq M}\mathfrak{g}^e\arrow{l}{\varphi}\arrow{dl}{\psi}\\
 v_\mu\otimes V_\mu^*
\end{tikzcd}
\end{equation}
and it is now equivalent to show that 
\[\text{ker}(\varphi)=\text{ker}(\psi).\]
Since diagram \eqref{second} is commutative, $\text{ker}(\varphi)\subseteq\text{ker}(\psi).$ Conversely, if $x\in\text{ker}(\psi)$, then
\[\psi(x)=\sum_{i} v_\mu^*(x\cdot v_i) v_\mu\otimes v_i^*=0\]
and so $v_\mu^*(x\cdot v_i)=0$ for each $v_i$. But then $x\cdot v_\mu^*=0$ and so by Lemma \ref{ann} the element $x$ annihilates every $v^*\in V_\mu^*$, so $x\in\text{ker}(\varphi)$. 

Thus $\text{ker}(\varphi)=\text{ker}(\psi)$, and in diagram \eqref{first} 
\[\text{coker}(\varphi^*)=\text{coker}(\psi^*).\]
In other words, for any $v^*\otimes u\in V_\mu^*\otimes V_\mu$, there is a $w\in V_\mu$ such that for any $x\in \mathcal{U}\mathfrak{g}^e$,
\[v^*(x\cdot u)=v_\mu^*(x\cdot w).\]

From this we can obtain the same result for functions on the group $G$. Because $G^e$ is a unipotent group and $V_\mu$ is a finite-dimensional representation, it is a general fact that $\varphi(G^e)\subset\varphi(\mathcal{U}\mathfrak{g}^e)$. So for any $g\in G^e$ we have
\[v^*(g\cdot u)=v_\mu^*(g\cdot w).\qedhere\]
\end{proof}

\begin{rk}
\label{4.5}
Proposition \ref{toprow} tells us that the ideal $\mathcal{I}_{\cent}$ contains, in each graded component 
\[\left(\bigoplus_{\mu\leq n\lambda}V^*_\mu\otimes V_\mu\right)t^{n\lambda},\]
all elements of the form 
\[(f^\mu_{v^*,u}-f^\mu_{v_\mu^*,w})t^{n\lambda}\]
for $v^*$, $u$, and $w$ as above. 
\end{rk}

We prove two more results that partially reverse the correspondence in Proposition \ref{toprow} and that will be useful in Section \ref{centralizers}.

\begin{lem}
\label{annihilators}
Let $v^*\in V_\mu^*$ be such that $v^*(v_\mu)\neq0$. Then
\[\text{\emph{Ann}}_{\mathcal{U}\mathfrak{g}^e}(v_\mu^*) =\text{\emph{Ann}}_{\mathcal{U}\mathfrak{g}^e}(v^*).\]
\end{lem}
\begin{proof} From Lemma \ref{ann} there is an inclusion,
\[\iota: \text{Ann}_{\mathcal{U}\mathfrak{g}^e}(v_\mu^*) \longhookrightarrow \text{Ann}_{\mathcal{U}\mathfrak{g}^e}(v^*).\]
As in the proof of Proposition \ref{toprow}, the space $V^*_\mu$ is $\mathbb{Z}$-graded and the algebras $\mathcal{U}\mathfrak{g}^e$ and $\text{Ann}_{\mathcal{U}\mathfrak{g}^e}(v_\mu^*)$ are $\mathbb{N}$-graded by the eigenvalues of $h$. For $m\in\mathbb{N}$, the collection
\[\mathcal{U}^{\geq m}\mathfrak{g}^e:=\bigoplus_{i\geq m} \mathcal{U}^i\mathfrak{g}^e.\]
is a decreasing filtration that induces decreasing filtrations on both $\text{Ann}_{\mathcal{U}\mathfrak{g}^e}(v_\mu^*)$ and $\text{Ann}_{\mathcal{U}\mathfrak{g}^e}(v^*)$. It is sufficient to show that the induced map 
\[\text{gr}\iota: \text{gr}(\text{Ann}_{\mathcal{U}\mathfrak{g}^e}(v_\mu^*))=\text{Ann}_{\mathcal{U}\mathfrak{g}^e}(v_\mu^*) \longhookrightarrow \text{gr}(\text{Ann}_{\mathcal{U}\mathfrak{g}^e}(v^*))\]
on associated graded algebras is surjective.

Let $k$ be the degree of $v_\mu^*$ under the grading---this is the minimal eigenvalue of $h$ on $V_\mu^*$. Then we have 
\[v^*=cv_\mu^*+w^*\]
for a nonzero constant $c$ and for $w^*$ sitting in degrees strictly higher than $k$. Let 
\[x\in\text{gr}(\text{Ann}_{\mathcal{U}\mathfrak{g}^e}(v^*))_n\subset\mathcal{U}^{\geq m}\mathfrak{g}^e/\mathcal{U}^{\geq m+1}\mathfrak{g}^e\] be nontrivial, with representative $x'\in \mathcal{U}^{\geq m}\mathfrak{g}^e$. We write 
\[x'=x^{(m)}+x'',\]
where $x^{(m)}$ is a nonzero element in degree $m$ and $x''$ sits in degree strictly higher than $m$. Then
\[0=x'\cdot v^*=cx'\cdot v_\mu^*+x'\cdot w^*=cx^{(m)}\cdot v_\mu^*+cx''\cdot v_\mu^*+x^{(m)}\cdot w^*+x''\cdot w^*.\]
The first term has degree $m+k$, and all the other terms sit in strictly higher degrees, so we must have
\[x^{(m)}\cdot v_\mu^*=0.\]
Therefore, $x^{(m)}\in\text{Ann}_{\mathcal{U}\mathfrak{g}^e}(v_\mu^*)$ is such that 
\[\text{gr}\iota\left(x^{(m)}\right)=x.\qedhere\]
\end{proof}

\begin{prop}
\label{bijection}
Let $w\in V_\mu$, and let $v^*$ be as in Lemma \ref{annihilators}. Then there exists an element $u\in V_\mu$ such that for any $g\in G^e$,
\[v_\mu^*(g\cdot w)=v^*(g\cdot u).\]
\end{prop}
\begin{proof}
Let $\varphi$ be the representation map from the proof of Proposition \ref{toprow}, and consider the restriction $\varphi_{res}^*$ of $\varphi^*$ to $v^*\otimes V_\mu$. 
\begin{equation*}
\begin{tikzcd}
  v^*\otimes V_\mu\arrow{r}{\varphi_{res}^*}&(\mathcal{U}^{\leq M}\mathfrak{g}^e)^*\\
 v_\mu^*\otimes V_\mu\arrow{ur}{\psi^*}
\end{tikzcd}
\end{equation*}
We would like to show that 
\[\text{Im}(\varphi_{res}^*)=\text{Im}(\psi^*).\]

The first inclusion already follows from \ref{toprow}, and to show the second it is sufficient to show that  
\[\text{ker}(\varphi_{res})\subset\text{ker}(\psi)\]
in the following diagram, where $v\in V_\mu$ is the dual vector to $v^*$ under the choice of weight vector basis in the proof of Proposition \ref{toprow}:
\begin{equation*}
\begin{tikzcd}
  v\otimes V_\mu^*&\mathcal{U}^{\leq M}\mathfrak{g}^e\arrow{l}{\varphi_{res}}\arrow{dl}{\psi}\\
 v_\mu\otimes V_\mu^*
\end{tikzcd}
\end{equation*}
If $x\in\text{ker}(\varphi_{res})$, then 
\[\varphi_{res}(x)=\sum_{i} v^*(x\cdot v_i) v\otimes v_i^*=0\]
and so $v^*(x\cdot v_i)=0$ for each $v_i$. Then $x\cdot v^*=0$ and by Lemma \ref{annihilators} the element $x$ annihilates $v_\mu^*$, so $x\in\text{ker}(\psi)$.

As in the proof of Proposition \ref{toprow}, since $\text{Im}(\varphi_{res}^*)=\text{Im}(\psi^*)$, it follows that there is an element $u\in V_\mu$ such that 
\[v_\mu^*(g\cdot w)=v^*(g\cdot u)\qquad\text{for any $g\in G^e$.}\qedhere\]
\end{proof}

Next we will show that if $\mu\leq\lambda$, then every function $f^\mu_{v_\mu^*,w}\in v_\mu^*\otimes V_\mu$ on $G^e$ is equivalent to a function $f^\lambda_{v_\lambda^*,z}$. For this we will need a result similar to Corollary \ref{ann}, and it will follow from the same theorem of Ginzburg:

\begin{lem}
\label{ginzburgnew}
Let $\mu\leq\lambda$ be dominant weights. Then
\[\text{\emph{Ann}}_{\mathcal{U}\mathfrak{g}^e}(v_\lambda^*)\subseteq \text{\emph{Ann}}_{\mathcal{U}\mathfrak{g}^e}(v_\mu^*).\]
\end{lem}
\begin{proof}
The orbits of $\check{G}(\mathbb{C}[[t]])$ on the affine Grassmannian $\check{G}(\mathbb{C}((t)))/\check{G}(\mathbb{C}[[t]])$ are indexed by the dominant weights of $G$, and since $\mu\leq\lambda$, we have
\[\mathbb{O}_\mu\subset\overline{\mathbb{O}}_\lambda.\]
(See Theorem 2.17 in \cite{schmitt}.)

The induced restriction map on cohomology
\[H^{\bullet}(\overline{\mathbb{O}}_{\lambda},\mathbb{C})\longrightarrow H^{\bullet}(\overline{\mathbb{O}}_{\mu},\mathbb{C})\]
is surjective since $\overline{\mathbb{O}_\lambda}$ and $\overline{\mathbb{O}_\mu}$ have compatible decompositions into affine strata. In view of Theorem \ref{affgrass}, this gives a surjection
\[\mathcal{U}\mathfrak{g}^e/\text{Ann}_{\mathcal{U}\mathfrak{g}^e}(v_\lambda^*)\longrightarrow\mathcal{U}\mathfrak{g}^e/\text{Ann}_{\mathcal{U}\mathfrak{g}^e}(v_\mu^*),\]
and implies that 
\[\text{Ann}_{\mathcal{U}\mathfrak{g}^e}(v_\lambda^*)\subseteq \text{Ann}_{\mathcal{U}\mathfrak{g}^e}(v_\mu^*).\qedhere\]
\end{proof}

\begin{prop}
\label{telescope}
Let  $\mu\leq\lambda$ and $v_\mu^*\otimes w\in V_\mu^*\otimes V_\mu$. Then there is an element $z\in V_\lambda$ such that for any $g\in G^e$,
\[v_\mu^*(g\cdot w)=v_\lambda^*(g\cdot z).\]
\end{prop}
\begin{proof}
As in the proof of Proposition \ref{toprow}, we will show this first for linear functions on the universal enveloping algebra. Consider the following representation maps:
\begin{align*}
&\varphi_\mu:\mathcal{U}\mathfrak{g}^e\longrightarrow v_\mu\otimes V_\mu^*\\
&\varphi_\lambda:\mathcal{U}\mathfrak{g}^e\longrightarrow v_\lambda\otimes V_\lambda^*
\end{align*}
As in the previous proof, we can restrict to considering 
\begin{align*}
&\varphi_\mu:\mathcal{U}^{\leq M}\mathfrak{g}^e\longrightarrow v_\mu\otimes V_\mu^*\\
&\varphi_\lambda:\mathcal{U}^{\leq M}\mathfrak{g}^e\longrightarrow v_\lambda\otimes V_\lambda^*
\end{align*}
for some sufficiently large integer $M$. The dual maps realize the elements of $v_\mu^*\otimes V_\mu$ and of $v_\lambda^*\otimes V_\lambda$ as functions on the universal enveloping algebra
\begin{equation}
\label{firstprime}
\begin{tikzcd}
  v_\mu^*\otimes V_\mu\arrow{r}{\varphi_\mu^*}&(\mathcal{U}^{\leq M}\mathfrak{g}^e)^*\\
 &v_\lambda^*\otimes V_\lambda\arrow{u}{\varphi_\lambda^*},
\end{tikzcd}
\end{equation} 
and we would like to show that every function in $v_\mu^*\otimes V_\mu$, when restricted to $\mathcal{U}\mathfrak{g}^e$, is equivalent to a function in $v_\lambda^*\otimes V_\lambda$.

Therefore, we will prove that the image of $\varphi_\mu^*$ is contained in the image of $\varphi_\lambda^*$. Dualizing diagram \eqref{firstprime},
\begin{equation}
\label{secondprime}
\begin{tikzcd}
  v_\mu\otimes V_\mu^*&\mathcal{U}^{\leq M}\mathfrak{g}^e\arrow{l}{\varphi_\mu}\arrow{d}{\varphi_\lambda}\\
 &v_\lambda\otimes V_\lambda^*
\end{tikzcd}
\end{equation}
it is equivalent to show that 
\[\text{ker}(\varphi_\lambda)\subseteq\text{ker}(\varphi_\mu).\]

Suppose $x\in\text{ker}(\varphi_\lambda)$. Then
\[\varphi_\lambda(x)=\sum_{i}v_\lambda^*(x\cdot v_i)v_\lambda\otimes v_i^*=0\]
and so $v_\lambda^*(x\cdot v_i)=0$ for each $v_i$. But then $x\cdot v_\lambda^*=0$, and so by Lemma \ref{ginzburgnew} we also have $x\cdot v_\mu^*=0$, and therefore $x\in\text{ker}(\varphi_\mu)$.

So $\text{ker}(\varphi_\lambda)\subseteq\text{ker}(\varphi_\mu)$, and therefore $\text{Im}(\varphi_\lambda^*)\supseteq\text{Im}(\varphi_\mu^*)$. For any $v_\mu^*\otimes w\in v_\mu^*\otimes V_\mu$, there is an element $z\in V_\lambda$ such that for any $x\in\mathcal{U}^{\leq M}\mathfrak{g^e}$,
\[v_\mu^*(x\cdot w)=v_\lambda^*(x\cdot z).\]
As in the proof of Proposition \ref{toprow}, it follows that
\[v_\mu^*(g\cdot w)=v_\lambda^*(g\cdot z)\qquad\text{for any $g\in G^e$.}\qedhere\]
\end{proof}

\begin{rk}
\label{4.11}
Proposition \ref{telescope} implies that the ideal $\mathcal{I}_{\cent}$ also contains all elements of the form
\[(f^{n\lambda}_{v_\lambda^*,z}-f^\mu_{v_\mu^*,w})t^{n\lambda}\in\left(\bigoplus_{\mu\leq n\lambda}V^*_\mu\otimes V_\mu\right)t^{n\lambda}\]
for $w$ and $z$ as above. We are now ready to prove Theorem \ref{mainthm}.
\end{rk}

\begin{proof}[Proof of Theorem \ref{mainthm}]
First note that by \eqref{topmult}, the function $\Phi'$ defined in \eqref{phi'}  is a homomorphism of graded algebras. We have $R_\lambda[\pet]=R_\lambda[\mathcal{B}]/\mathcal{I}_{\pet}$, where
\begin{align*}
\mathcal{I}_{\pet}&=\bigoplus_{n\geq0}\left\{u\in V_{n\lambda}\mid v_{n\lambda}^*(g\cdot u)=0,\, \forall gB^-\in\pet\right\}\\
	&=\bigoplus_{n\geq0}\left\{u\in V_{n\lambda}\mid v_{n\lambda}^*(g\cdot u)=0,\, \forall g\in G^e\right\},
\end{align*}
since the image of $G^e$ is dense in $\pet$.
Similarly $R_\lambda[\cent]=R_\lambda[\overline{G}]/\mathcal{I}_{\cent},$ where
\begin{align*}
\mathcal{I}_{\cent}&=\bigoplus_{n\geq0}\left\{\sum f^\mu_{v^*,u}t^{n\lambda}\in \bigoplus_{\mu\leq n\lambda} V_{\mu}^*\otimes V_{\mu}\mid \sum f^\mu_{v^*,u}(g)\lambda(t)^n=0,\,\forall (g,t)\in G^e\times T\right\}\\
	&=\bigoplus_{n\geq0}\left\{\sum f^\mu_{v^*,u}t^{n\lambda}\in \bigoplus_{\mu\leq n\lambda} V_{\mu}^*\otimes V_{\mu}\mid \sum f^\mu_{v^*,u}(g)=0,\,\forall g\in G^e\right\},
\end{align*}
since the function $t^{n\lambda}=\lambda(t)^n$ is always nonzero.

We will check everything on graded components. First, $\Phi'$ does indeed descend to a homomorphism of algebras 
\[\Phi: R_\lambda[\pet]\longrightarrow R_\lambda[\cent],\]
since for any $u\in\mathcal{I}_{\pet}\cap V_{n\lambda}$ and $(g,t)\in G^e\times T$
\begin{align*}
\Phi'(u)(g,t)&=f^{n\lambda}_{v_{n\lambda}^*,u}(g)\lambda(t)^n=v_{n\lambda}^*(g\cdot u)\lambda(t)^n=0,
	\end{align*}
and so $\Phi'(u)\in\mathcal{I}_{\cent}$.

Second, the homomorphism $\Phi$ is injective: if $\Phi'(u)\in\mathcal{I}_{\overline{G^e}}$ for some $u\in V_{n\lambda}$, then
\[v_{n\lambda}^*(g\cdot u)\lambda(t)^n=0\]
for all $(g,t)\in G^e\times T$, so $u\in\mathcal{I}_{\pet}$. Thus, $\text{ker}(\Phi')=\mathcal{I}_{\pet}$, and $\text{ker}(\Phi)=0$.

Last, $\Phi$ is surjective: suppose $f^\mu_{v^*,u}t^{n\lambda}\in (V_\mu^*\otimes V_\mu)t^{n\lambda}$.  By Proposition \ref{toprow}, there is a $w\in V_\mu$ such that 
\[f^\mu_{v^*,u}t^{n\lambda}\equiv f^\mu_{v_\mu^*,w}t^{n\lambda} \text{ (mod }\mathcal{I}_{\cent}),\]
as noted in Remark \ref{4.5}.
By Proposition \ref{telescope} there is a $z\in V_{n\lambda}$ such that
\[f^\mu_{v_\mu^*,w}t^{n\lambda}\equiv f^{n\lambda}_{v_\lambda^*,z}t^{n\lambda} \text{ (mod }\mathcal{I}_{\cent}),\]
as in Remark \ref{4.11}.
Then
\[\Phi(z)\equiv f^\mu_{v^*,u}t^{n\lambda}\text{ (mod }\mathcal{I}_{\cent}).\qedhere\]
\end{proof}

\section{The General Case}
\label{centralizers}
Now let $x\in\mathfrak{g}$ be a regular element, not necessarily nilpotent, and let $G^x\subset G$ be its centralizer. By the Jordan decomposition and by conjugating appropriately,
\[x=s+n\]
for some semisimple $s\in \mathfrak{t}$ and a nilpotent $n\in \mathfrak{n}$ such that 
\[n=\sum_{i\in I}e_i\]
is a sum of the simple root vectors indexed by the set $I\subset\{1,\ldots,l\}$.

The centralizer of $s$ in the group $G$ is the centralizer of the one-parameter subgroup \mbox{$\{\text{exp}(ts)\mid t\in\mathbb{C}^*\}$} and is therefore a Levi subgroup $L\subset G$ (see \cite{dig.mic:91}, Proposition 1.22). The centralizer $G^x=L^n$, being abelian, decomposes as
\[G^x= C\times A,\]
where $C$ is the center of $L$ and $A=[L,L]^n\cap N$ is the unipotent part of the centralizer of $n$ in the derived subgroup $[L,L]$. The element $n$ is a principal nilpotent of $[L,L]$, and $A$ is a unipotent subgroup that centralizes it, so all the results from Section \ref{main} apply to $A$ as a subgroup of the semisimple group $[L,L]$. (See Remark \ref{choices}.)

For any dominant weight $\lambda$ of $G$, the irreducible representation $V_\lambda$ decomposes into irreducible representations of $L$
\[V_\lambda\simeq\bigoplus_{(\alpha,\rho)\in[\lambda]}W^\alpha_\rho,\]
where $W^\alpha_\rho$ is the irreducible representation of $[L,L]$ of highest weight $\rho$ with an action of $C$ by the character $\alpha$, and $[\lambda]$ denotes the set of pairs $(\alpha, \rho)$ that appear in the decomposition of $V_\lambda$. Let $w_\rho^\alpha$ denote the highest weight vector of $W^\alpha_\rho$.

As before, fix a regular dominant weight $\lambda$ in the root lattice of $G$, and let $V_\lambda^*$ be the dual of the corresponding representation. There is a decomposition
\[V_\lambda^*\simeq\bigoplus_{(\alpha,\rho)\in[\lambda]}W^{\alpha*}_\rho,\]
and we denote the lowest weight vector of $W^{\alpha*}_\rho$ by $w^{\alpha*}_\rho$.

The dominant weight $\lambda$ gives rise to the line bundle $\mathcal{L}_\lambda$ on $\mathcal{B}$, with space of global sections
\[\Gamma(\mathcal{B},\mathcal{L}_\lambda)=V_\lambda\]
as in Section \ref{peterson}.
\begin{definition}
An element $\mathfrak{b}\in\mathcal{B}$ is \emph{general} if for all $(\alpha,\rho)\in[\lambda]$,
\[w^\alpha_\rho(\mathfrak{b})\neq0,\]
where $w^\alpha_\rho\in V_\lambda$ is viewed as a global section of $\mathcal{L}_\lambda$. The $G^x$-orbit of such an element is a \emph{general orbit} of $G^x$.
\end{definition}

Generality is independent of the basepoint of a $G^x$-orbit, and it is an open and nonempty condition. Let $h\in G$ be such that the $h$-translate $h\cdot\mathfrak{b}^-$ is general. This is the case if and only if
\[(h\cdot v_\lambda^*)(w_\rho^\alpha)\neq0\]
for all $(\alpha,\rho)\in[\lambda]$, and then $h\cdot v_\lambda^*$ satisfies the condition of Lemma \ref{annihilators} and Proposition \ref{bijection}.

Let $\px$ be the closure of the (general) $G^x$-orbit of $h\cdot\mathfrak{b}^-$ in $\mathcal{B}$, and let $\overline{G^x}$ be the closure of $G^x$ in the wonderful compactification $\overline{G}$. We will use the methods of Section \ref{main} to show that the varieties $\px$ and $\overline{G^x}$ are isomorphic. 

Consider the homogeneous coordinate rings of the flag variety and of the wonderful compactification given by the projective embeddings corresponding to $\lambda$. As before, we have
\begin{align*}
&R_\lambda[\px]=R_\lambda[\mathcal{B}]/\mathcal{I}_{\px}\\
&R_\lambda[\overline{G^x}]=R_\lambda[\overline{G}]/\mathcal{I}_{\overline{G^x}}
\end{align*}
where $\mathcal{I}_{\px}$ and $\mathcal{I}_{\overline{G^x}}$ are the ideals of global sections that vanish on $\px$ and $\overline{G^x}$ respectively. 

Define, component-wise, a map 
\begin{align}
\label{psi'}
\Psi': \,R_\lambda[\mathcal{B}]&\longrightarrow R_\lambda[\overline{G}]\\
	u&\longmapsto (h\cdot v_{n\lambda}^*\otimes u)t^{n\lambda}\nonumber
	\end{align}
for $u\in V_{n\lambda}$. We will show

\begin{thm}
\label{general}
The map $\Psi'$ descends to an isomorphism of graded algebras 
\[\Psi: R_\lambda[\px]\longrightarrow R_\lambda[\overline{G^x}].\]
\end{thm}

\begin{prop}
\label{toprowgeneral}
Let $w^*\otimes u\in W^{\alpha*}_\rho\otimes W^{\alpha}_\rho$. There exists an element $v\in W^\alpha_\rho$ such that for all $g\in G^x$
\[w^*(g\cdot u)=w^{\alpha*}_\rho(g\cdot v).\]
\end{prop}
\begin{proof}[Proof of the Proposition]
We decompose the centralizer $G^x$ as 
\[G^x\simeq C\times A.\]
When we restrict $W^\alpha_\rho$ to $[L,L]\subset L$ the representation remains irreducible, and by Proposition \ref{toprow} there is a $v\in W^\alpha_\rho$ such that for any $a\in A$
\[w^*(a\cdot u)=w^{\alpha*}_\rho(a\cdot v).\]
We can write any $g\in G^x$ as $g=ca$ with $c\in C$ and $a\in A$, and since $C$ acts on $W^\alpha_\rho$ by $\alpha$ we have
\begin{align*}
w^*(g\cdot u)&=w^*(ca\cdot u)\\
			&=\alpha(c)w^*(a\cdot u)\\
			&=\alpha(c)w^{\alpha*}_\rho(a\cdot v)\\
			&=w^{\alpha*}_\rho(ca\cdot v)=w^{\alpha*}_\rho(g\cdot v).\qedhere
\end{align*}
\end{proof}

Proposition \ref{toprowgeneral} is an analogue to Proposition \ref{toprow}. Proposition \ref{telescopegeneral} will give an analogous result to Proposition \ref{telescope}, and the following lemma will allow us to apply it to the proof of Theorem \ref{general}. 

We introduce an new item of notation: If two integral weights $\theta$ and $\xi$ of $T$ differ by a linear combination of simple roots of $[L,L]$ with positive integral coefficients, we will write $\theta\leq_L\xi$ to indicate that $\theta$ is less than $\xi$ in the partial ordering on the weight lattice of $[L,L]$.

\begin{lem}
\label{leviirreps}
Suppose $\mu$ and $\lambda$ are dominant weights of $G$ such that $\mu\leq\lambda$. Then for any $(\alpha,\sigma)\in[\mu]$ there exists a dominant weight $\rho$ of $[L,L]$ such that $\sigma\leq_L\rho$ and $(\alpha,\rho)\in[\lambda]$.
\end{lem}
\begin{proof}
Let $\text{Spec}(\mu)$ and $\text{Spec}(\lambda)$ denote the set of all weights of $G$ that appear in the irreducible representations $V_\mu$ and $V_\lambda$ respectively. Since $\mu\leq\lambda$, $\text{Spec}(\mu)\subset\text{Spec}(\lambda)$. (See \cite{fultonharris}, Section 14.1.)

If $(\alpha,\sigma)\in[\mu]$, then $\alpha+\sigma\in\text{Spec}(\mu)\subset\text{Spec}(\lambda)$, so there is some $(\beta,\rho)\in[\lambda]$ such that $\alpha+\sigma$ appears as a weight in $W^\beta_\rho$. 

Since the center $C$ acts by the same character on all of $W^\beta_\rho$, we must have $\beta=\alpha$. When we restrict the representation $W^\alpha_\rho$ to the derived subgroup $[L,L]$, it is the irreducible representation of $[L,L]$ of highest weight $\rho$. Since $\sigma$ appears as a weight in this representation, $\sigma\leq_L\rho$.
\end{proof}

\begin{prop}
\label{telescopegeneral}
Let $\sigma$ and $\rho$ be dominant weights of $[L,L]$ such that $\sigma\leq_L\rho$, and let $\alpha$ be a character of $C$. Let $v\in W^\alpha_\sigma.$ Then there exists an element $z\in W^\alpha_\rho$ such that for all $g\in G^x$,
\[w^{\alpha*}_\sigma(g\cdot v)=w^{\alpha*}_\rho(g\cdot z).\]
\end{prop}
\begin{proof}
Since $\sigma\leq_L\rho$, by Proposition \ref{telescope} there exists an element $z\in W^{\alpha*}_\rho$ such that for any $a\in A$,
\[w^{\alpha*}_\sigma(a\cdot v)=w^{\alpha*}_\rho(a\cdot z).\]
Then we can write any $g\in G^x$ as $g=ca$ with $c\in C$ and $a\in A$, and since $C$ acts by the character $\alpha$ on both $W^{\alpha*}_\sigma$ and $W^{\alpha*}_\rho$ we have
\begin{align*}
w^{\alpha*}_\sigma(g\cdot v)&=w^{\alpha*}_\sigma(ca\cdot v)\\
						&=\alpha(c)w^{\alpha*}_\sigma(a\cdot v)\\
						&=\alpha(c)w^{\alpha*}_\rho(a\cdot z)\\
						&=w^{\alpha*}_\rho(ca\cdot z)=w^{\alpha*}_\rho(g\cdot z).\qedhere
\end{align*}
\end{proof}

As in Section \ref{wond}, we will use the notation $f^{\alpha,\sigma}_{w^*,v}$ to denote a global section arising from an element $w^*\otimes v\in W^{\alpha*}_\sigma\otimes W^\alpha_\sigma.$

\begin{proof}[Proof of Theorem \ref{general}]
As before, by \eqref{topmult} the function $\Psi'$ is a homomorphism of graded algebras. We have $R_\lambda[\px]=R_\lambda[\mathcal{B}]/\mathcal{I}_{\px}$, where
\begin{align*}
\mathcal{I}_{\px}&=\bigoplus_{n\geq0}\left\{u\in V_{n\lambda}\mid v_{n\lambda}^*(h^{-1}g\cdot u)=0,\forall g^{-1}hB^-\in\px\right\}\\
&=\bigoplus_{n\geq0}\left(\bigoplus_{(\alpha,\rho)\in[n\lambda]}\left\{v\in W^{\alpha}_{\rho}\mid  (h\cdot v^*_{n\lambda})(g\cdot v)=0,\forall g\in G^x\right\}\right).
\end{align*}

Moreover, $R_\lambda[\overline{G^x}]=R_\lambda[\overline{G}]/\mathcal{I}_{\overline{G^x}}$, where
\begin{align*}
\begin{split}
\mathcal{I}_{\overline{G^x}}&=\bigoplus_{n\geq0}\left(\bigoplus_{\mu\leq n\lambda}\left\{\sum f^\mu_{v^*,u}\in V^*_\mu\otimes V_\mu\mid \sum f^\mu_{v^*,u}(g)\lambda(t)^n=0,\forall (g,t)\in G^x\times T\right\}\right)\\
&=\bigoplus_{n\geq0}\Bigg(\bigoplus_{\mu\leq n\lambda}\bigoplus_{(\alpha,\sigma)\in[\mu]}\left\{\sum f^{\alpha,\sigma}_{w^*,v}\in W^{\alpha*}_\sigma\otimes W^\alpha_\sigma\mid\right.\\
&\qquad\qquad\qquad\qquad\qquad\qquad\qquad\left. \sum f^{\alpha,\sigma}_{w^*,v}(g)\lambda(t)^n=0,\forall (g,t)\in G^x\times T\right\}\Bigg).
\end{split}
\end{align*}

We will check everything on graded components. The homomorphism $\Psi'$ does indeed descend to a homomorphism of algebras
\[\Psi:R_\lambda[\px]\longrightarrow R_\lambda[\overline{G^x}],\]
since for any $u\in\mathcal{I}_{\px}$ and $(g,t)\in G^x\times T$,
\[\Psi'(u)(g,t)=f^{n\lambda}_{h\cdot v_{n\lambda}^*, u}t^{n\lambda}(g,t)=v_{n\lambda}^*(h^{-1}g\cdot u)\lambda(t)^n=0.\]

Moreover, $\Psi$ is injective: if $\Psi'(u)=0$ for some $u\in V_{n\lambda}$, then
\[v_{n\lambda}^*(h^{-1}g\cdot u)\lambda(t)^n=0\]
for all $(g,t)\in G^x\times T$, so $u\in\mathcal{I}_{\px}$. So $\text{ker}(\Psi')=\mathcal{I}_{\px}$, and $\text{ker}(\Psi)=0$.

Lastly, $\Psi$ is surjective. We will prove this first in degree $1$---since the homogeneous coordinate ring is generated in degree $1$, surjectivity will then follow for all degrees.

Suppose $\mu\leq\lambda$, $(\alpha,\sigma)\in[\mu]$, and $f^{\alpha,\sigma}_{w^*,v}t^{\lambda}\in (W^{\alpha*}_\sigma\otimes W^\alpha_\sigma)t^{\lambda}$. By Proposition \ref{toprowgeneral}, there is a $u\in W^\alpha_\sigma$ such that
\[f^{\alpha,\sigma}_{w^*,v}t^{\lambda}\equiv f^{\alpha,\sigma}_{w^{\alpha*}_\sigma,u}t^{\lambda} \text{ (mod }\mathcal{I}_{\overline{G^x}}).\]
By Lemma \ref{leviirreps}, there is some dominant weight $\rho$ of $[L,L]$ such that $\sigma\leq_L\rho$ and $(\alpha,\rho)\in [\lambda]$, and by Proposition \ref{telescopegeneral} there is an element $y\in W^\alpha_\rho$ such that
\[f^{\alpha,\sigma}_{w^{\alpha*}_\sigma,u}t^{\lambda}\equiv f^{\alpha,\rho}_{w^{\alpha*}_\rho,y}t^{\lambda} \text{ (mod }\mathcal{I}_{\overline{G^x}}).\]
Let $\pi_\rho^\alpha:V^*_\lambda\to W_\rho^{\alpha*}$ denote the projection of $V_\lambda^*$ onto $W_\rho^{\alpha*}$. Then $\pi_\rho^\alpha(h\cdot v_\lambda^*)(w_\rho^\alpha)\neq0$, so by Proposition \ref{bijection} there is an element $z\in W_\rho^\alpha$ such that 
\[f^{\alpha,\rho}_{w^{\alpha*}_\rho,y}t^{\lambda}\equiv f^{\alpha,\rho}_{\pi_\rho^\alpha(h\cdot v_\lambda^*),z}t^{\lambda} \text{ (mod }\mathcal{I}_{\overline{G^x}}).\]

Then
\begin{align*}
\Psi'(z)=\left(h\cdot v_\lambda^*\otimes z\right)t^\lambda=\left(\pi_\rho^\alpha(h\cdot v_\lambda^*)\otimes z\right)t^\lambda,
\end{align*}
so in fact
\[\Psi(z)\equiv f^{\alpha,\sigma}_{w^*,v}t^{\lambda} \text{ (mod }\mathcal{I}_{\overline{G^x}}).\qedhere\]
\end{proof} 

\section{Orbits on the Peterson Variety}
\label{orbit}
We will give a description of the orbits of $G^e$ on $\pet$, and in particular we will show that in most cases there are infinitely many. For this we will consider the Peterson variety as a subvariety of the flag variety $G/B$ with basepoint $\mathfrak{b}$, coming from the embedding
\begin{align*}
G^e&\longhookrightarrow G/B\\
g&\longmapsto gw_0\cdot \mathfrak{b}
\end{align*}
where $w_0$ is the longest word of the Weyl group $W$.

Let $f_1,\ldots,f_l$ be the negative simple root vectors in $\mathfrak{g}$. In this setting, the Peterson variety has the following description \cite{rietsch}:
\begin{align}
\label{peter}
\pet=\left\{gB\in G/B\mid Ad(g^{-1})\cdot e\in\mathfrak{b}\oplus\left(\sum_{i=1}^l\mathbb{C}f_i\right)\right\}.
\end{align}

We introduce some notation. For any $I\subseteq\{1,\ldots,l\}$ indexing a subset of the simple roots $\{\alpha_i\mid i\in I\}$, let $P_I$ be the corresponding parabolic subgroup, $L_I$ its Levi subgroup, $U_I$ its unipotent radical, and $\mathfrak{l}_I=\text{Lie}(L_I)$ and $\mathfrak{u}_I=\text{Lie}(U_I)$ their Lie algebras. 

Let $N_I\subset L_I$ be the maximal unipotent subgroup of the Levi, and $\mathfrak{n}_I$ its Lie algebra. Let $W_I$ be the subgroup of $W$ generated by the reflections corresponding to the simple roots $\{\alpha_i\mid i\in I\}$, and let $w_I$ be the longest element of $W_I$. Let 
\[e_I=\sum_{i\in I} e_i\]
be a nilpotent element of $\mathfrak{g}$, and note that it is regular in $[\mathfrak{l}_I,\mathfrak{l}_I]$. 

The centralizer of $e_I$ in $L_I$ decomposes as a product
\[L_I^{e_I}=C_I\times A_I\]
where $C_I=Z(L_I)$ is the center of $L_I$ and $A_I$ is a unipotent subgroup of $L_I$, as in Section \ref{centralizers}.

To find the $G^e$-orbits on $\pet$, we will use the Bruhat decomposition. We have 
\[\pet=\bigcup_{w\in W} \left(\pet\cap NwB/B\right)\]
and each intersection $\pet\cap NwB/B$ is a $G^e$-stable subset. 

\begin{lem}[\cite{har.tym:10}, Proposition 5.8]
The intersection of $\pet$ with the Schubert cell $NwB/B$ non-empty if and only if $w$ is the longest word $w_I$ of some parabolic Weyl group $W_I$.
\end{lem}
\begin{proof}
Suppose $nwB\in NwB/B$ is in the Peterson variety. Then by \eqref{peter}
\[w^{-1}n^{-1}\cdot e\in \mathfrak{b}\oplus \sum_{i=1}^l\mathbb{C}f_i.\]
We can write
\[n^{-1}=\text{exp}(x)\]
for some nilpotent $x\in\mathfrak{n}$, and then
\begin{align}
\label{longestword}
w^{-1}n^{-1}\cdot e&=w^{-1}\left(e+x\cdot e+\frac{x^2\cdot e}{2}+\ldots\right)\nonumber\\
&=w^{-1}\cdot e+w^{-1}\left(x\cdot e+\frac{x^2\cdot e}{2}+\ldots\right).
\end{align}
Since the two summands in \eqref{longestword} belong to disjoint sums of roots spaces, this means in particular that 
\[w^{-1}\cdot e\in\mathfrak{b}\oplus\sum_{i=1}^l\mathbb{C}f_i.\]
So for any simple root $\alpha_i$, $w^{-1}\cdot\alpha_i$ is either a simple negative root or a positive root. But this precisely characterizes the longest words of parabolic Weyl groups (see \cite{rie3}, Lemma 2.2), and so $w=w_I$ for some $I\subseteq\{1,\ldots,l\}$.
\end{proof}

\begin{rk}
The following result is proved by Insko and Yong \cite{inskoyong} in type $A$, and is known to experts in the general case.
\end{rk}

\begin{prop}
In the notation above,
\[\pet\cap\, Nw_IB/B=A_Iw_IB/B.\]
\end{prop}
\begin{proof}
Suppose first that $h\in A_I$, so that $h$ centralizes $e_I$, and write $e=e_I+e_I'$, where 
\[e_I'=\sum_{i\notin I}e_i\in \mathfrak{u}_I.\]
We will show that $hw_IB\in\pet$. Using \eqref{peter}, we obtain
\begin{align*}
w_Ih^{-1}\cdot e&=w_Ih^{-1}\cdot e_I+w_Ih^{-1}\cdot e_I'\\
&=w_I\cdot e_I+w_Ih\cdot e_I'
\end{align*}
Since $w_I$ negates all the positive roots $\{\alpha_i\mid i\in I\}$, the first term is in $\sum_{i\in I}\mathbb{C}f_i$. Since $\mathfrak{u}_I$ is normalized by $P_I$ and stable under the action of any representative of $w_I$, the second term is in $\mathfrak{u}_I$. So,
$$w_Ih^{-1}\cdot e\in\mathfrak{b}\oplus\sum_{i=1}^l\mathbb{C}f_i$$
and $hw_IB\in\pet.$

Conversely, let $n\in N$ so that $nw_IB\in\pet.$ Then 
\[w_In^{-1}\cdot e\in\mathfrak{b}\oplus\sum_{i=1}^l\mathbb{C}f_i.\]
Decomposing $e$ as above,
\[w_In^{-1}\cdot e=w_In^{-1}\cdot e_I+w_In^{-1}\cdot e_I',\]
and as before the second term is in $\mathfrak{u}_I$, so in fact
\[w_In^{-1}\cdot e_I\in\mathfrak{b}\oplus\sum_{i=1}^l\mathbb{C}f_i.\]

By the Levi decomposition, $n^{-1}=vu$ with $v\in L_I$ and $u\in U_I$. Since $n\in N$, we have $v\in N_I$, so we can write $v=\text{exp}(x)$ and $u=\text{exp}(y)$ for $x\in\mathfrak{n}_I$ and $y\in\mathfrak{u}_I$. Then
\begin{align*}
w_In^{-1}\cdot e_I&=w_I\text{exp}(x)\text{exp}(y)\cdot e_I\\
&=w_I(e_I+x\cdot e_I+\text{terms in $\mathfrak{u}_I$}).
\end{align*}
In particular,
\begin{align}
\label{wherex}
w_Ix\cdot e_I\in\mathfrak{b}\oplus\sum_{i=1}^l\mathbb{C}f_i.
\end{align}
But $x$ is a sum of positive root vectors of strictly positive height in the levi $\mathfrak{l}_I$, so $x\cdot e_I$ is a sum of root vectors with root height at least $2$. That is,
\[ x\cdot e_I\in\mathfrak{l}_I\cap\left(\bigoplus_{\text{ht}(\alpha)\geq2}\mathfrak{g}_\alpha\right),\]
and since $w_I$ flips every root in $\mathfrak{l}_I$, 
\[ w_Ix\cdot e_I\in\mathfrak{l}_I\cap\left(\bigoplus_{\text{ht}(\alpha)\leq-2}\mathfrak{g}_\alpha\right),\]
and to satisfy \eqref{wherex} we must have $x\cdot e_I=0$. Then $h=\text{exp}(x)\in A_I$ and
\[nw_IB=vuw_IB=vw_IB\in A_Iw_IB\]
since $u\in U_I$ and $U_I$ is $w_I$-stable. 
\end{proof}

In particular, $A_Iw_IB/B$ is $G^e$-stable, being the intersection of two $G^e$-stable subvarieties of $G/B$. The following Proposition describes the $G^e$-orbits on $A_Iw_IB/B$ . Define 
$$\pi_I:P_I\longrightarrow L_I$$
to be the projection of the parabolic $P_I$ onto its Levi subgroup. The image of $G^e$ under this projection centralizes $e_I$, because $e_I$ is itself the image of $e$ under the differential $d\pi_I:\mathfrak{p}_I\longrightarrow \mathfrak{l}_I.$ Therefore,
\[\pi_I(G^e)\subset A_I.\]

\begin{prop}
\label{orbits}
The $G^e$-orbits on $A_Iw_IB/B=\pet\cap Nw_IB/B$ are in bijection with the cosets of $A_I/\pi_I(G^e)$.
\end{prop}
\begin{proof}
Let $h,k\in A_I$ and suppose first that $hw_I\in gkw_IB$ for some $g\in G^e$. Then
\[k^{-1}g^{-1}h\in w_IBw_I\]
and in fact
\[k^{-1}g^{-1}h\in w_IBw_I\cap B=U_I.\]
The Levi decomposition gives $g^{-1}=xu$ for $x=\pi_I(g^{-1})\in A_I$ and $u\in U_I$, and we can write
\[k^{-1}g^{-1}h=k^{-1}xuh=k^{-1}xhh^{-1}uh=k^{-1}xhu'\]
where $u'\in U_I$ since $P_I$ normalizes $U_I$. Since this expression is in $U_I$, we have
\[k^{-1}xh\in U_I\]
and since $k,x,h\in L_I$ we conclude
\[k^{-1}xh=1.\]
Thus $k=xh$ and the elements $h$ and $k$ of $A_I$ are $\pi_I(G^e)$-translates. 

Conversely, suppose that $k=xh$ for some $x\in\pi_I(G^e)$. Then there is some $u\in U_I$ such that $xu\in G^e$, and we have
\begin{align*}
xu\cdot(hw_IB)&=kh^{-1}uhw_IB\\
&=kh^{-1}hvw_IB\quad\text{for some $v\in U_I$, since $h$ normalizes $U_I$}\\
&=kw_IB\quad\text{since $w_I$ normalizes $U_I$}
\end{align*}
so the cosets $hw_IB$ and $kw_IB$ are in the same $G^e$-orbit.
\end{proof}

Proposition \ref{orbits} gives a bijective correspondence between the $G^e$-orbits on the intersection of the Peterson variety with the Schubert cell $Nw_IB$ and the $\pi_I(G^e)$-cosets in the subgroup $A_I$ of $L_I$. Since $A_I$ and $\pi_I(G^e)$ are unipotent groups, the coset space $A_I/\pi_I(G^e)$ is in fact a vector space.

Because the dimension of $\pi_I(G^e)$ may be strictly less than the dimension of $A_I$, there may be infinitely many $G^e$-orbits in the boundary of the Peterson variety. In type A this is the case for all choices of $I$ for which $[\mathfrak{l}_I,\mathfrak{l}_I]$ is not simple, and such a choice exists in all ranks strictly greater than $2$. 

\bibliographystyle{plain}
\bibliography{final}

\end{document}